\documentclass[11pt]{amsart}
\usepackage{bm}

\usepackage{mathpazo}
\usepackage{latexsym}   
\usepackage{amssymb}    
\usepackage{amsmath}    
\usepackage{amsthm}     
\usepackage{hyperref}

\newtheorem{theoremintro}{Theorem}

\numberwithin{equation}{subsection}
\newtheorem{theorem}{Theorem}[section]
\newtheorem{lemma}[theorem]{Lemma}
\newtheorem{proposition}[theorem]{Proposition}
\newtheorem{corollary}[theorem]{Corollary}

\theoremstyle{definition}
\newtheorem{remark}[theorem]{Remark}

\DeclareMathOperator{\orth}{O}
\DeclareMathOperator{\po}{PO}
\DeclareMathOperator{\Lie}{Lie}

\def\good{{\rm good}}
\def\stable{{\rm st}}
\def\semistable{{\rm ss}}
\def\ball{{\mathbb B}}

\newcommand{\stk}[1]{{\mathcal #1}}
\newcommand{\tilstk}[1]{{\til{\stk #1}}}
\newcommand{\barstk}[1]{{\bar{\stk #1}}}
\newcommand{\coarse}[1]{{\underline{#1}}}

\usepackage{diagrams}
\newarrow{to}---->
\newarrow{dashedto}{}{dash}{}{dash}>
\newarrow{mto}|--->
\newarrow{impto}===={=>}
\newarrowtail{<=}\Rightarrow\Leftarrow\Uparrow\Downarrow
\newarrow{bboth}{<=}==={=>}
\newarrow{inject}{hooka}---{vee}
\newarrow{iline}{hooka}----
\newarrow{surject}----{>>}

\def\marked{{\rm m}}
\def\weyl{{\mathbb W}}
\def\setcomp{\smallsetminus}

\def\cx{{\mathbb C}}
\def\ff{{\mathbb F}}
\def\proj{{\mathbb P}}
\def\rat{{\mathbb Q}}

\def\integ{{\mathbb Z}}

\def\zz{{\bm 0}}

\def\idp{{\frak p}}

\def\call{{\mathcal L}}
\def\calz{{\mathcal Z}}
\def\cala{{\mathcal A}}

\def\calf{{\mathcal F}}
\def\caln{{\mathcal N}}
\def\calo{{\mathcal O}}
\def\calt{{\mathcal T}}
\def\calx{{\mathcal X}}
\def\calp{{\mathcal P}}
\def\calm{{\mathcal M}}
\def\caly{{\mathcal Y}}

\newcommand{\til}[1]{{\widetilde{#1}}}
\newcommand{\st}[1]{\{#1\}}
\newcommand{\ang}[1]{{{\langle #1 \rangle}}}

\newcommand{\sra}[1]{\stackrel{#1}{\rightarrow}}
\newcommand{\rest}[1]{|_{#1}}

\DeclareMathOperator{\id}{id}
\DeclareMathOperator{\aut}{Aut}
\DeclareMathOperator{\gal}{Gal}
\DeclareMathOperator{\SL}{SL}

\DeclareMathOperator{\End}{End}
\DeclareMathOperator{\tr}{tr}
\DeclareMathOperator{\fil}{F}
\DeclareMathOperator{\pic}{Pic}
\DeclareMathOperator{\Frac}{Frac}

\def\ra{\rightarrow}
\def\tensor{\otimes}
\def\iso{\cong}
\def\cross{\times}
\def\inject{\hookrightarrow}
\def\bs{\backslash}
\def\sriso{\stackrel{\sim}{\rightarrow}}
\def\sep{{{\rm sep}}}
\def\derham{{\rm dR}}
\def\cris{{\rm cris}}
\def\units{^\cross}
\DeclareMathOperator{\spec}{Spec}
\def\comp{\circ}
\def\inv{^{-1}}
\def\twiddle{\sim}

\newenvironment{alphabetize}{\begin{enumerate}

}{\end{enumerate}}

\begin{document}

\title{Arithmetic Torelli maps for cubic surfaces and threefolds}

\author{Jeffrey D. Achter}
\email{j.achter@colostate.edu}
\address{Department of Mathematics, Colorado State University, Fort
Collins, CO 80523} 
\urladdr{http://www.math.colostate.edu/~achter}

\thanks{This work was partially supported by a grant from the Simons
  foundation (204164).}

\begin{abstract}
  It has long been known that to a complex cubic surface or threefold
  one can canonically associate a principally polarized abelian
  variety.  We give a construction which works for cubics over an
  arithmetic base, and in particular identifies the moduli space of
  cubic surfaces with an open substack of a certain moduli space of
  abelian varieties.  This answers, away from the prime $2$, an old
  question of Deligne and a recent question of Kudla and Rapoport.
\end{abstract}

\maketitle

\section{Introduction}

Consider a complex cubic surface.  By associating to it either
a K3 surface \cite{dolgachevetal05} or a cubic threefold
\cite{actcras,act02}, one can construct a polarized Hodge
structure.  The Hodge structures which arise are
parametrized by $\ball^4$, the complex $4$-ball, and in this way
one can show that the moduli space $\stk S_\cx$ of complex cubic
surfaces is uniformized by $\ball^4$.

It turns out that the relevant arithmetic quotient of $\ball^4$ is the
set of complex points of $\stk M$, the moduli space of abelian
fivefolds with action by $\integ[\zeta_3]$ of signature $(4,1)$, and
there is a Torelli map $\tau_\cx:\stk S_\cx \ra \stk M_\cx$.  The main
goal of the present paper is to construct an arithmetic Torelli map
for cubic surfaces:

\begin{theoremintro}
\label{tha}
There is an open immersion
\begin{diagram}
\stk S & \rto & \stk M
\end{diagram}
of stacks over $\integ[\zeta_3,1/6]$ which specializes to $\tau_\cx$.
\end{theoremintro}
We refer to Proposition \ref{propperiodsurface}(b) and Theorem
\ref{thbigdiagram} for more details.  In particular, this immersion
extends to a homeomorphism if $\stk S$ is replaced with its partial
compactification $\stk S_\stable$, the moduli stack of stable cubic
surfaces.   Theorem \ref{tha} shows that $\stk S_\stable$ is
essentially a Shimura variety, and thus supports unexpected arithmetic
structures, such as Hecke 
correspondences and modular forms.  

It turns out that a key step in understanding the Torelli map for
cubic surfaces is to show
that there is an arithmetic construction which recovers the
intermediate Jacobian of a complex cubic threefold.  Let $\tilstk T
\subset H^0(\proj^4, \calo_{\proj^4}(3))$ be the scheme of smooth cubic
forms in five variables, and let $\jmath_\cx:\tilstk T_\cx \ra
\cala_{5,\cx}$ be the intermediate Jacobian functor.

\begin{theoremintro}
\label{thb}
There is a morphism 
\begin{diagram}
\tilstk T & \rto & \cala_5
\end{diagram}
of stacks over $\integ[1/2]$ which specializes to $\jmath_\cx$.
\end{theoremintro}
(See Theorem \ref{thbuildprym} and Corollary \ref{corperiodthreefold} below.)

Both of these results are statements about arithmetic period maps, but
with slightly different antecedents.   Theorems \ref{thb} and
\ref{tha} provide special cases of conjectures of, respectively,
Deligne on algebraic period maps and of Kudla and Rapoport on occult
period maps.  

\subsection{Algebraic period maps}

It is a classical insight that
one can study a complex variety by examining its periods. For
a Riemann surface, the cokernel of the period lattice is not just a 
complex analytic torus but an algebraizable one.  Constructing this 
abelian variety through algebro-geometric, as opposed to
transcendental, means was a central concern of abstract algebraic
geometry.  More recently, although still classically, one understood
how to make sense of this construction in families.  (See
\cite{kleimansurvey} for a comprehensive survey.)

Similarly, Clemens and Griffiths \cite{clemensgriffiths} show that a
complex cubic threefold is characterized by its middle cohomology.
The intermediate Jacobian of such a variety
is a principally polarized abelian fivefold, and Clemens and Griffiths
prove a global Torelli theorem.  Mumford gives an algebro-geometric
construction (via Prym varieties) of the intermediate Jacobian, and
thus reveals the essential algebraic nature of this {\em a priori}
transcendental object.  Beauville and Murre extend these ideas to
cubic threefolds over algebraically closed fields of (almost)
arbitrary characteristic.  Theorem \ref{thb} states that these
constructions make sense not just pointwise but for families over a
general base.  Thus, it provides a new example of an algebraic
period map.  (In fairness, one should note that the essential
algebraization step is provided by the existence of Picard schemes of
relative curves.)


In \cite{deligneniveau}, Deligne studies the period map, via
intermediate Jacobians, for smooth complete intersections of Hodge
level one.  He proves that the universal intermediate Jacobian
attached to the complex fiber descends to $\rat$, and conjectures that
it spread out to $\integ$.  Theorem \ref{thb} is apparently the first
nontrivial case of Deligne's conjecture (away from the prime $2$).  In
joint work with Casalaina-Martin, the author is currently
investigating the extent to which the techniques of the present work
allow the resolution of other cases of Deligne's conjecture.  (The key
point is that a cubic threefold is, after suitable blowup, a quadric
fibration; examination of the list \cite{rapoportniveau} of complete
intersections of Hodge level one indicates that similar progress may
be made in other cases.)

\subsection{Occult period maps}
\label{subsecoccult}
The moduli space of stable complex cubic surfaces is an arithmetic
quotient of the complex $4$-ball.  Thanks to work of, e.g., Allcock,
Carlson and Toledo \cite{act3folds}, Looijenga and Swierstra
\cite{looijengaswierstra}, Kond{\=o} and
Dolgachev and Kond{\=o} \cite{dolgachevkondo,kondog4}, several other examples are known of
apparently accidental isomorphisms between certain moduli spaces of
complex varieties and (open subsets of) arithmetic quotients of complex
unit balls.  As part of their work on unitary Shimura varieties
\cite{kudlarapoportii}, Kudla and Rapoport survey such occult period maps
(``occult'', because their construction relies not on the periods of
the varieties in question, but on other, auxiliary varieties)
\cite{kroccult}, and show that they descend to certain cyclotomic
number fields $\rat(\zeta_n)$.  They conjecture \cite[Conj.\ 15.9]{kudlarapoportii} that each
such descends to $\integ[\zeta_n,1/n]$.  Thus (away from $2$), Theorem
\ref{tha} secures one case of the conjecture of Kudla  
and Rapoport.  The construction given here depends on the fact that
one knows how to construct the Picard scheme of a relative curve; the
algebraicity of $\tau$ comes from that of the classical Torelli map
for curves. 

\subsection{Alternate approaches}

Like \cite{act02}, the present work constructs the occult
period map $\tau$ by associating, to each cubic surface, the cyclic
cubic threefold ramified along that surface, and then constructing an
associated abelian variety.

There are other methods available for demonstrating an exceptional
isomorphism between complex moduli spaces and ball quotients, and they
may well yield novel approaches to and extensions of Theorem \ref{tha}.

In \cite{delignemostow}, the authors show that certain moduli spaces
of weighted points in $\proj^1$ are ball quotients, by using the
period map for cyclic covers of the projective line.  Some of these
results have been recovered by instead considering periods of K3
surfaces.  Many of the occult period maps reviewed in \cite{kroccult},
and indeed Dolgachev, van Geemen and Kond{\=o}'s
construction of $\tau_\cx$, go through a construction involving K3
surfaces.  Recent work on arithmetic period maps for K3 surfaces
may well allow progress on the problem of arithmetic occult periods;
the author hopes to return to this subject shortly.

\subsection{Outline}

After establishing notation in Section \ref{secmoduli}, in Section
\ref{secprym} we construct a period morphism $\til\varpi:\tilstk T
\ra \stk A_5$ from the space of smooth cubic forms in five variables
to the moduli space of abelian fivefolds.  This is used in Section
\ref{secsurfaces} to define the morphism of
$\integ[\zeta_3,1/6]$-stacks $\tau:\stk S \ra \stk M$.   (The reader
is invited to consult diagram \ref{diagbig} to see how these morphisms
fit together.)
In Section
\ref{secimage}, we characterize the image of $\tau$ and ultimately
extend the domain of $\tau$ to the moduli space of semistable cubic
surfaces.  

The existence of an open immersion $\stk S \inject \stk M$ over
$\integ[\zeta_3,1/6]$ means that information about the arithmetic and
geometry of abelian varieties can be transported to the setting of
cubic surfaces.  In a companion work \cite{achtercubmono}, the
author exploits this 
connection in order to classify the abelian varieties which arise as
intermediate Jacobians of cubic threefolds attached to cubic surfaces.

\subsection*{Acknowledgments}
It's a pleasure to thank, especially, M.\  Rapoport and D.\ Toledo for
their encouragement; S.\ Casalaina-Martin for helpful conversations,
particularly concerning the boundary of the moduli space of cubic
threefolds; and the referee for insightful suggestions.

\section{Moduli spaces}
\label{secmoduli}

Let $\calo_E$ be the ring of Eisenstein integers $\integ[\zeta_3]$,
and let $E = \rat(\zeta_3)$.
In Sections \ref{subsecmodcubics} and
\ref{secprym}, all constructions are over $\integ[1/2]$; elsewhere,
all spaces are objects over $\calo_E[1/6]$.  The present section
establishes notation concerning moduli spaces of cubic surfaces and
threefolds (\ref{subsecmodcubics}) and  abelian varieties (\ref{subsecmodabvars}).

\subsection{Cubics}
\label{subsecmodcubics}
Let $\tilstk S \subset H^0(\proj^3, \calo_{\proj^3}(3))\ra \spec
\integ$ be the space 
of smooth cubic forms in four variables.  Similarly, let $\tilstk
S_\stable$ and $\tilstk S_\semistable$ be the spaces of stable and
semistable cubic forms, respectively, in the sense of GIT \cite[Sec.\
4.2]{git}.   The complement $\tilstk S_\stable \setcomp \tilstk S$
is a horizontal divisor,  i.e.,
for each $t \in\spec \integ$, the fiber $(\tilstk
S_\stable\setcomp \tilstk S)_t$ has codimension one in $\tilstk
S_\stable$. 
Moreover, $\tilstk S_\semistable \setcomp \tilstk S_\stable$ forms a single orbit under
the action of $\SL_4$.

Let $\stk S = \tilstk S / \SL_4$ be the moduli stack of smooth projective
cubic surfaces, and similarly define $\stk S_\stable$ and $\stk
S_\semistable$.  A cubic surface is stable if it is either smooth or
has ordinary double points of type $A_1$ \cite[1.14]{mumfordens}.

The coarse moduli space $\coarse S_\semistable$ of
$\stk S_\semistable$ 
is a normal projective scheme.  Let $g:\caly\ra \stk S$ be the tautological cubic
surface over $\stk S$.  We will slightly abuse notation and also let
$\caly \ra \tilstk S$ be the universal cubic surface over $\tilstk S$,
thought of as a subscheme of $\proj^3_{\tilstk S}$.

We follow \cite[Sec.\ 2]{matsumototerasoma} for the notion of marked
cubic surface used here.  More precisely, we give a definition of
marked cubic surfaces which, in the special case of complex
cubic surfaces, coincides with the definition in
\cite{matsumototerasoma}.

Recall that if $Y/K$ is a cubic surface over an algebraically closed
field, then there are $27$ lines on $Y$.  Let $\Lambda(Y)$ be the
simple graph with a vertex $[L]$ for each line $L$ on $Y$, where $[L]$
and $[L']$ are adjacent if and only if $L$ and $L'$ have nontrivial
intersection.

It is known that $\Lambda(Y)$ is isomorphic to the abstract graph
$\Lambda_0$ defined as follows.  For $1 \le i \le 6$, there are
vertices $e_i$ and $c_i$, and for $1 \le j < k \le 6$ there is a
vertex $\ell_{jk}$.  Then $e_i$ and $c_j$ are adjacent if $i\not = j$;
each of $e_i$ and $c_i$ is adjacent to $\ell_{jk}$ if $i \in \st{j,k}$;
$\ell_{ij}$ and $\ell_{km}$ are adjacent if $\st{i,j}\cap \st{k,m} =
\emptyset$; and these are the only adjacencies in $\Lambda_0$.

A marking of $Y$ is a graph isomorphism $\Psi:\Lambda(Y) \ra
\Lambda_0$.  We briefly indicate how such a marking may be
constructed.  Realize $Y$ as the blowup $Y \ra \proj^2$ of the
projective plane at $6$ points $P_1, \cdots, P_6$.  Let $E_i$ be the
inverse image of $P_i$; for $1 \le i < j \le 6$ let $L_{ij}$ be the
strict transform of the line connecting $P_i$ and $P_j$; and let
$C_{i}$ be the strict transform of the conic through $\st{P_1, \cdots,
  P_6} \setcomp P_i$.  Then the rule $[E_i] \mapsto e_i$, $[L_{ij}]
\mapsto \ell_{ij}$, $[C_i] \mapsto c_i$ is a marking of $Y$.

The automorphism group of $\Lambda_0$ is $\weyl = W({\rm E}_6)$, the
Weyl group of the exceptional root system ${\rm E}_6$, and thus the set of
markings on a fixed cubic surface is a torsor under $\weyl$.
Moreover, an automorphism of a surface which fixes a marking is
necessarily the identity \cite[Prop.\ 1.1]{naruki}.  So, let $\stk S^\marked$ be the moduli
space of marked cubic surfaces $(Y,\Psi)$.  Then $\stk S^\marked$ is
(represented by) a smooth quasiprojective scheme, and under the
forgetful map we have $\stk S = \stk S^\marked/\weyl$.

Let $\stk S_\stable^\marked$ be the normalization of $\stk S_\stable$
in $\stk S^\marked$, and similarly define $\stk
S_\semistable^\marked$.  (In \cite{act02}, the complex fiber $\stk
S^\marked_{\stable,\cx}$ is constructed as a certain Fox completion; but
these two notions coincide \cite[8.1]{delignemostow}.) Then $\stk
S^\marked_\semistable$ is a normal projective scheme.

Finally, let $\tilstk T \subset H^0(\proj^4, \calo_{\proj^4}(3))$ be
the space of smooth cubic forms in five variables, and let $\stk T =
\tilstk T/\SL_5$ be the space of smooth projective cubic threefolds.  Let $h:\calz\ra
\stk T$ denote both the tautological cubic threefold over $\stk T$ and
its pullback to $\tilstk T$.  For the sake of completeness, we record
the fact that a cubic threefold is stable if and only if each of its
singularities has type $A_j$ with $j \le 4$.  (This is \cite[Thm.\
1.1]{allcock3folds}.  While the result there is only claimed for
complex threefolds, the proof relies only on the Hilbert-Mumford
stability criterion \cite[Thm.\ 2.1]{mumfordgit}, which is valid in
all characteristics \cite[Thm.\ 2.2]{seshadri72}.)

Taken together, we have parameter spaces $\tilstk S$ and $\tilstk T$
for smooth cubic forms in four and five variables, respectively.  Each
of these is an open subscheme of a projective space over $\spec
\integ$, and in particular Noetherian and smooth.  Taking GIT
quotients, we have smooth Deligne-Mumford stacks $\stk S$ and $\stk T$
over $\spec \integ$.  In the sequel we will consider these as objects
over $\spec \calo_E[1/6]$ and $\spec \integ[1/2]$, respectively,
and in particular set aside the question of how to mark a cubic surface in
characteristic three.

\subsection{Abelian varieties}
\label{subsecmodabvars}
We work in the category of $\calo_E[1/6]$-schemes.  In particular,
each scheme $S$ comes equipped with a morphism $\iota: \calo_E \ra
\calo_S$.  Let $\bar\iota: \calo_E \ra \calo_S$ be the composition of
the involution of $\calo_E$ and the morphism $\iota$.  Then
$\calo_S\tensor \calo_E \iso \calo_S\oplus \calo_S$, where $\calo_E$
acts on the first copy of $\calo_S$ via $\iota$ and on the second by
$\bar\iota$.  Under this decomposition, any locally free sheaf $\calf$
of $\calo_S\tensor \calo_E$-modules is a direct sum of two locally
free $\calo_S$-modules of ranks $r$ and $s$.  The pair $(r,s)$ is
called the signature of $\calf$.

Let $\stk M = \stk M_{(4,1)}$ be the moduli stack of principally
polarized abelian fivefolds equipped with an $\calo_E$-action of
signature $(4,1)$.  More precisely, $\stk M(S)$ consists of triples
$(X,\iota,\lambda)$, where $X \ra S$ is an abelian scheme of relative
dimension five; $\iota: \calo_E \ra \End_S(X)$ is an action of
$\calo_E$ such that $\Lie(X)$ has signature $(4,1)$; and $\lambda$ is
a principal polarization compatible with the action of $\calo_E$.  For
each prime $\ell$ invertible on $S$, the principal polarization
$\lambda$ induces a skew-symmetric unimodular form on the $\ell$-adic
Tate module, $\ang{\cdot,\cdot}_\lambda: T_\ell X \cross T_\ell X \ra
\integ_\ell(1)$.  The compatibility of $\lambda$ and $\iota$ may be
expressed by insisting that for $x,y \in T_\ell X$ and $a \in
\calo_E$, $\ang{\iota(a)x,y}_\lambda = \ang{x,
  \bar\iota(a)y}_\lambda$.  The forgetful morphism $\jmath:\stk M \ra
\stk A_5$ is an embedding, since an abelian fivefold admits at most
one $\integ[\zeta_3]$-action of signature $(4,1)$.  Let $f:\calx\ra
\stk A_5$ denote
both the tautological abelian scheme over $\stk A_5$ and its pullback
via $\jmath$ to $\stk M$.

For our construction of the Torelli map, we will need the notion of a
certain kind of level structure on an $\calo_E$-abelian variety.  For
a rational prime $\ell$, let $\calo_{E,\ell} = \calo_E \tensor
\integ_\ell$; then $\calo_{E,\ell}$ is a degree two
$\integ_\ell$-algebra with involution, and $T_\ell X$ is free of rank
five over $\calo_{E,\ell}$.  From general theory (e.g.,
\cite[2.3]{kudlarapoportii}), there exists a unique unimodular
$\calo_{E,\ell}$-Hermitian form $h_\lambda: T_\ell X \cross T_\ell X
\ra \calo_{E,\ell}(1)$ on $T_\ell X$ such that, for $x,y \in T_\ell
X$, $\ang{x,y}_\lambda =
\tr_{\calo_{E,\ell}/\integ_\ell}(h_\lambda(x,y)/\sqrt{-3})$.

Apply these considerations in the special case $\ell = 3$, and reduce
$h_\lambda$ modulo $(1-\zeta_3)$.  Note that $\calo_{E,3}/(1-\zeta_3)
\iso \calo_E/(1-\zeta_3) \iso \ff_3$; that $T_\ell X \tensor
\calo_{E,\ell}/(1-\zeta_3)$ is a five-dimensional vector space over
the residue field $\ff_3$; and that the choice of $\zeta_3$ yields an
isomorphism $(\calo_{E,\ell}/(1-\zeta_3))(1) \iso \ff_3$.  Suppose
that $X[1-\zeta_3]$, the $(1-\zeta_3)$-torsion subscheme of $X$, is
split over $S$, so that $X[1-\zeta_3](S) \iso
\ff_3^{\oplus 5}$.  (For an arbitrary $\calo_E$-abelian scheme, this
will only happen after \'etale extension of the base.)  Then $\lambda$
induces a perfect pairing 
\begin{diagram}
X[1-\zeta_3](S) \cross
X[1-\zeta_3](S) &\rto^{\bar h_\lambda}& \ff_3
.
\end{diagram}
Since the involution of $\calo_E$ is
trivial modulo $(1-\zeta_3)$, this Hermitian pairing is actually an
orthogonal form on $X[1-\zeta_3](S)$.

Let $(V_0, q_0)$ be a five-dimensional vector space over $\ff_3$
equipped with a perfect orthogonal form $q_0$.  A $(1-\zeta_3)$-level
structure on $(X,\iota,\lambda)$ is an isomorphism $\Phi:(V_0,q_0) \ra
(X[1-\zeta_3](S), \bar h_\lambda)$ of quadratic spaces; equivalently,
it is a choice of basis on $X[1-\zeta_3](S)$ which is orthonormal for
$\bar h_\lambda$.

Let $\stk M^{(1-\zeta_3)}$ be the moduli stack of principally polarized
$\calo_E$-abelian varieties with level $(1-\zeta_3)$ structure, i.e., of
quadruples $(X,\iota,\lambda,\Phi)$.  The usual Serre lemma shows that
any automorphism of $(X,\iota,\lambda)$ which fixes $\Phi$ is
necessarily the identity automorphism, so $\stk M^{(1-\zeta_3)}$ is a
smooth quasiprojective scheme.

The orthogonal group $\orth(V_0,q_0)$ acts on the set of
$(1-\zeta_3)$-level structures on $(X,\iota,\lambda)$.  Moreover, the
center of $\orth(V_0,q_0)$ acts trivially on such level structures
(e.g., \cite[Rem.\ 4.2]{matsumototerasoma}), so that the forgetful map
$\stk M^{(1-\zeta_3)} \ra \stk M$ has covering group
$\orth(V_0,q_0)/\st{\pm 1} = \po(V_0,q_0)$.  From the description of
$\weyl$ as the automorphism group of the Euclidean lattice ${\rm E}_6$, one
can extract an inclusion $\weyl \ra \orth(V_0,q_0)$ such that the
composition $\weyl \inject \orth(V_0,q_0) \ra \po(V_0,q_0)$ is an
isomorphism.

Let $\barstk M$ be the minimal (also known as the Satake, or
Baily-Borel) compactification of $\stk M$, and let 
$\barstk M^{(1-\zeta_3)}$ be the normalization of $\barstk M$ in $\stk
M^{(1-\zeta_3)}$.  The coarse moduli space $\coarse{\bar M}$ of
$\barstk M$ is a normal projective scheme, and $\barstk
M^{(1-\zeta_3)}$ is itself represented by a normal projective scheme.

\subsection{A remark on characteristic}

As is hopefully clear from the introduction, much of the work of the
present paper is to show that transcendental facts over $\cx$ have
their origins in algebraic facts over an arithmetic base.

In Section \ref{secprym}, we already find it necessary to invert the
prime $2$ in order to make an algebraic version of the  ``intermediate
Jacobian'' functor for cubic threefolds.  Initially, this is because
we make heavy use of results from \cite{beauville77} and
\cite{murre72a} on quadric fibrations, especially those coming from
cubic threefolds.  These results are valid over an arbitrary
algebraically closed field whose characteristic is not $2$.
Even if one were able to circumvent this difficulty, it is still not
clear that one can carry out the Prym construction in even
characteristic.   Poincar\'e reducibility is apparently not known for
arbitrary abelian schemes.  Thus, the invertibility of $2$ is used in
an essential way in Lemma \ref{lemprym}.

It is perhaps less surprising that, starting in Section
\ref{secsurfaces}, we find it expedient to invert $3$.  The initial
construction starts with a certain cyclic cover of degree $3$, and
then proceeds through a comparison involving abelian schemes with
$\integ[\zeta_3]$-action and level structure at $3$.  Each of these
notions makes sense in characteristic $3$, but is considerably more
delicate.

\section{Cubic threefolds}
\label{secprym}

After recalling the Prym construction of the intermediate Jacobian of
a cubic threefold over an algebraically closed field (Section
\ref{subsecprymalgclosed}), we use a descent argument to show that this
construction works for a cubic threefold over an arbitrary field
(Section \ref{subsecprymfield}), and even over an arbitrary normal
Noetherian base scheme (Section \ref{subsecprymscheme}). 
In
particular, we deduce (Corollary \ref{corperiodthreefold}) the
existence of an arithmetic Torelli map $\varpi: \tilstk T \ra \stk
A_5$.  As discussed below, this proves a special case of
\cite[3.3]{deligneniveau}.   
The main construction of this section requires a theory of relative
Prym schemes, which is worked out in Section \ref{subsecprym}

In this section, we work in the category of schemes over
$\integ[1/2]$.  In particular, we will only consider fields whose
characteristic is either zero or odd.

\subsection{Algebraically closed fields}
\label{subsecprymalgclosed}

Let $Z/\cx$ be a cubic threefold.  One can associate to $Z$ the
intermediate Jacobian $J(Z) = \fil^2 H^3(Z,\cx)\bs
H^3(Z,\cx)/H^3(Z,\integ)$, a principally polarized abelian fivefold.
The map $\til \calt(\cx) \ra \cala_{5}(\cx)$ comes from a complex
analytic map $\til \calt_\cx \ra \cala_{5,\cx}$.  

Alternatively,
$J(Z)$ can be constructed without recourse to transcendental methods,
using a Prym construction.  Briefly, let $L \subset Z$ be a suitably
generic line (in the sense of, say, \cite[Prop.\ 1.25]{murre72a}).
Projection to the space of planes through $L$ gives $Z^*$, the
blowup of $Z$ along $L$, a structure of fibration $\varpi_L: Z^* \ra \proj^2$. The fibers of
$\varpi_L$ are conics, which are smooth outside the discriminant locus
$\Delta_L \subset \proj^2$.  Direct calculation (using the genericity of
$L$) shows that $\Delta_L$
is a smooth, irreducible curve of genus $6$.  Moreover, inside the
Fano surface $F_Z$ of lines on $Z$ there is an \'etale double cover
$\til\Delta_L$ of $\Delta_L$, necessarily of genus $11$.  There is a
norm map $\caln:\pic^0(\til \Delta_L) \ra \pic^0(\Delta_L)$ of
principally polarized abelian varieties.  Let $P_L$ be the connected component of identity of the 
kernel of $\caln$.  By Mumford's theory of the Prym variety, $P_L$ is
an abelian variety which inherits a canonical principal polarization
from that of $\pic^0(\til \Delta_L)$.  It turns out that $J(Z)$ is
isomorphic to $P_L$ as a principally polarized abelian variety.

The Prym construction is already well-understood over 
algebraically closed fields of any characteristic (other than two)
\cite{beauville77,murre72a}.  The goal of the present section is to show
that this construction makes sense in families.

Thanks to Beauville, Mumford and Murre, one knows that Prym varieties
over algebraically closed fields have a certain universal property.
This property is briefly recalled here, following \cite[Sec.\ 3.2]{beauville77}.

Let $Z/k$ be a cubic threefold over an algebraically closed field in
which $2$ is invertible, and 
as above let $Z^*$ be the blowup of $Z$ along a good line $L$.  Let
$A^2(Z)$ be the group of rational
equivalence classes of cycles on $Z$ of codimension two which are
algebraically equivalent to zero, and define $A^2(Z^*)$ analogously.

There is an isomorphism of abstract groups $P_L(k) \sriso A^2(Z^*)$
\cite[Thm.\ 3.1]{beauville77}.
Inverting this map gives an isomorphism of abstract groups
$\alpha'_L(k): A^2(Z^*)\sriso P_L(k)$
which is regular, as follows.
For a smooth variety $T/k$, declare that a (set-theoretic) map
$a:T(k)\ra A^2(Z^*)$ is algebraic if there 
exists a cycle $\xi$ on $Z^*\cross T$ such that, for $t\in T(k)$,
$\xi_t = a(t)$. Then $\alpha'_L(k)$ is regular in the sense that for
any algebraic map $T(k) \ra A^2(Z^*)$, the composition $T(k) \ra
A^2(Z^*) \ra P_L(k)$ is induced by a morphism $T \ra P_L$ of
varieties.  Following Beauville, we will abuse notation somewhat and
write this morphism as $\alpha'_L: A^2(Z^*) \ra P_L$.

Beauville shows that $P_L$ is the universal algebraic representative of $A^2(Z)$: 

\begin{proposition}
\label{propuniversal}
\begin{alphabetize}
\item There is a regular map $\alpha_L: A^2(Z) \ra P_L$ which induces
  an isomorphism $A^2(Z) \iso P_L(k)$.
\item If $X/k$ is any abelian variety and
$A^2(Z) \ra X$ is regular, then there is a unique morphism $\beta_X:P_L \ra
X$ which makes the following diagram commute:
\begin{diagram}[LaTeXeqno]
\label{diaguniversalredux}
A^2(Z) & &\rto&& X \\
&\rdto^{\alpha_L} &&\ruto^{\beta_X} \\
&& P_L &&
\end{diagram}
\item There is a correspondence on $A^2(Z)$ which induces the
  principal polarization on $P_L$; in \eqref{diaguniversalredux}, if
  $A^2(Z) \ra X$ is compatible with this correspondence, then
  $\beta_X$ is a morphism of polarized abelian varieties.
\end{alphabetize}
\end{proposition}

\begin{proof}
There is a canonical isomorphism of abstract
groups $A^2(Z) \iso A^2(Z^*)$ induced by the blowup morphism $b_L:Z^*
\ra Z$ \cite[Lem.\ 2]{murre73}.  Suppose $T(k) \ra A^2(Z)$ is
algebraic, and represented by a cycle $\xi$ on $Z\cross T$.  Then the
composition $T(k) \ra A^2(Z) \sra{b_L^*} A^2(Z^*)$ is represented by
the cycle $(b_L^*\cross\id)(\xi)$, and thus algebraic.  There is thus a
regular map $\alpha_L: A^2(Z) \ra P_L$.
Part (b)  is precisely \cite[Prop.\ 3.3]{beauville77}; $P_L$ is
the unique algebraic representative of $A^2(Z)$.  The compatibility
with polarizations is \cite[Prop.\ 3.5]{beauville77}.  (We emphasize
that the cited work of Beauville is valid over any algebraically
closed field which is not of characteristic two.)
\end{proof}

In particular, the principally polarized abelian variety $P_L$ is, up
to canonical isomorphism, 
independent of the choice of (suitably generic) $L$, and 
we will denote this abelian variety by $P(Z)$.

\subsection{Arbitrary fields}
\label{subsecprymfield}

As a warmup we will show that, given a cubic threefold $Z$ over an
arbitrary field $K$, one can canonically associate to it a principally
polarized abelian fivefold over $K$.

\begin{lemma}
\label{lemarbfield}
Let $Z/K$ be a cubic threefold. 
\begin{alphabetize}
\item If $L$ is a sufficiently generic line on $Z_{\bar K}$, then
$P_L(Z_{\bar K})$ descends to $K$.
\item If $M$ is a second sufficiently generic line on $Z_{\bar K}$,
  then there is a canonical isomorphism $P_L(Z) \ra P_M(Z)$ of
  principally polarized abelian varieties over $K$.
\end{alphabetize}
\end{lemma}

\begin{proof}
  Let $F = F_Z/K$ be the variety of lines on $Z$.  Then $F$ is an
  irreducible smooth (Fano) surface (e.g., \cite[1.1]{murre72a}, which
  is valid both in characteristic zero and in positive, odd
  characteristic).  Let $F^\good\subset F$ be the dense open
  subvariety constructed in \cite[Prop.\ 1.25]{murre72a}; it is
  defined over $K$.  Choose a finite Galois extension $K'/K$ such that
  there is a line $L$ (with moduli point) in $F^\good(K')$.
  (Membership in $F^\good$ is the ``suitable genericity'' referred to
  in Section \ref{subsecprymalgclosed}.)  The projection
  $\varpi_L:Z^*_{K'} \ra \proj^2$ is defined over $K'$, and thus so is
  its discriminant locus $\Delta_L\subset \proj^2$.  In fact,
  $\Delta_L$ is a smooth, projective, absolutely irreducible curve of
  genus $6$.

Inside $F$, let $\til\Delta_L$ be the set of (moduli points of) lines
which meet $L$; then $\til\Delta_L$ is a smooth, projective absolutely
irreducible curve \cite[Prop.\ 1.25.(iv)]{murre72a}.  Since $L$ and
$F$ are defined over $K'$, so is $\til\Delta_L$.  Moreover, there is a
natural \'etale double cover $\til\Delta_L \ra \Delta_L$; the fiber
over a point of $\Delta_L$ consists of the two lines of the degenerate
conic fiber over it.  The Jacobians $\pic^0(\til\Delta_L)$ and
$\pic^0(\Delta_L)$ are also defined over $K'$, and the Prym variety $P_L$ is a
principally polarized abelian variety over $K'$ (Lemma
\ref{lemprym}). We will use Galois
  descent (e.g., \cite[Sec.\ 
  6.2.B]{blr} or \cite[Thm.\ 6.23]{milneag}) to show that these
  objects are actually defined over $K$.

  Suppose $\sigma \in \gal(K'/K)$.  Then $\sigma L := L
  \cross_{K',\sigma} K'$ is also in $F^\good(K')$, and $\sigma P_L$
  may be calculated as $\sigma P_L = P_{\sigma L}$.

  Let $N \ge 3$ be invertible in $K$; if necessary, enlarge $K'$ so
  that $K'$ is still Galois over $K$, but contains the field of
  definition of all $N$-torsion of each $P_{\sigma L}$.  (Note that
  this does not enlarge the set of conjugates of $P_L$ under
  $\gal(K^\sep/K)$.)

  The universal property of Prym varieties (Proposition \ref{propuniversal})
  implies that, for $\sigma \in \gal(K'/K)$, there is a canonical
  isomorphism $\beta_{\sigma,\bar K}: P_{\sigma L, \bar K} \sriso P_{L,
    \bar K}$ of principally polarized abelian varieties over $\bar K$
  compatible with the maps $\alpha_{\sigma L, \bar K}:A^2(Z_{\bar K})
  \ra P_{\sigma L, \bar K}$ and $\alpha_{L,\bar K}:A^2(Z_{\bar K}) \ra
  P_{L, \bar K}$.  Since $K'$ includes the field of definition of the
  $N$-torsion of each $P_{\sigma L}$, the isomorphism
  $\beta_{\sigma,\bar K}$ is actually defined over $K'$ \cite[Thm.\
  2.4]{silverberg92}.  Moreover, because $P_{\sigma L} = \sigma P_L$,
  we obtain canonical isomorphisms $\beta_\sigma: \sigma P_L \sriso P_L$.
  In particular, if $\sigma, \tau \in \gal(K'/K)$, then (again by
  Proposition \ref{propuniversal}) the following diagram commutes:
\begin{diagram}
  \sigma \tau P_L & & \rto^{\beta_{\sigma\tau}} && P_L \\
  &\rdto<{\sigma \beta_\tau} &&\ruto>{\beta_\sigma} \\
  && \sigma P_L
\end{diagram}
Therefore, the data $\st{ \beta_\sigma: \sigma \in \gal(K'/K)}$
defines a $K'/K$ descent datum on $P_L$. Because $P_L$ is projective,
this descent datum is effective, and $P_L$ is defined over $K$.  This
proves (a).

The second part follows again from Galois descent.  As above, let $K'$
be a finite, Galois extension of $K$ such that $L$, $M$ and the
$N$-torsion of the associated Pryms are all defined over $K'$.  Then
for each $\sigma \in \gal(K'/K)$ there is a canonical isomorphism
$P_{\sigma L}(Z_{\bar K}) \sriso P_{\sigma M}(Z_{\bar K})$; this system
descends to the desired isomorphism $P_L(Z) \sriso P_M(Z)$.
\end{proof}

Consequently, if $Z/K$ is a cubic threefold over a field, one may
canonically associate to it a principally polarized abelian fivefold,
which will be simply denoted $P(Z)$.

\begin{remark}
Murre has shown \cite[Thm.\ 7]{murre74} that $P(Z_{\bar K})$ is
isomorphic to the Albanese variety of $F_{Z_{\bar K}}$.  This, combined
with the general fact that an Albanese variety exists over an
arbitrary base field, gives an alternative proof of Lemma
\ref{lemarbfield}.  However, the techniques used here will be deployed
again over a more general base scheme in Theorem \ref{thbuildprym},
while relative Albanese schemes are not known to exist in full
generality.
\end{remark}

\subsection{Arbitrary base}
\label{subsecprymscheme}

In fact, the Prym construction makes sense for
a cubic threefold over an arbitrary normal Noetherian base.

\begin{theorem}
\label{thbuildprym}
Let $S$ be a normal Noetherian scheme over $\integ[1/2]$, and let $Z
\ra S$ be a relative smooth cubic threefold.  There is a principally
polarized abelian scheme $P(Z) \ra S$ such that, for each
point $s \in S$, $P(Z)_s \iso P(Z_s)$.
\end{theorem}

\begin{proof}
  By definition, $Z$ comes equipped with an embedding inside a
  $\proj^4$-bundle $\proj V$ over $S$.  Because $S$ is locally
  Noetherian we may consider $F \ra S$, the scheme of
  lines on $Z$ \cite[Thm.\ 3.3]{altmankleiman}.  It is a smooth 
  relative surface over $S$.

Let $F^\good \subset F$ be the subscheme such that, for $s\in S$,
$(F^\good)s \iso F^\good_{Z_s}$, where the latter is the nonempty,
open subvariety constructed in \cite[Prop.\ 1.25]{murre72a} and used
in Lemma \ref{lemarbfield}.

Not only is every fiber of $F^\good \ra S$ dense in the corresponding
fiber of $F \ra S$, but $F^\good$ is actually open in $F$.  We briefly
sketch why this is so.  It suffices to verify the claim in the case
where $S = \spec A$ and $\proj V \iso \proj^4_S$.
For a given point $s\in S$, Murre characterizes
$F^\good_{Z_s}$ by first constructing certain open subsets $U_{1,s}$
and $U_{2,s}$, given as the complement of the vanishing locus of
certain explicit equations \cite[eq.\ (11) and Lemma 1.11]{murre72a} involving a cubic form
defining $Z_s$;
 then imposing an open condition
$U_{3,s}$, that a smooth hyperplane section of $Z_s$ passes through
the line in question; and then realizing $F^\good_{Z_s}$ as the
intersection of $U_{1,s}$, $U_{2,s}$ and $U_{3,s}$.  In fact, for each
$i \in \st{1,2,3}$, there is an open $U_i \subset F^\good$ whose fiber
at $s$ is $U_{i,s}$.  For $i=3$, the Bertini-type condition is clearly
open, while for $i = 1,2$, one can simply use Murre's
equations, now with coefficients in $A$.

  Initially, suppose that there is a section of $F^\good\ra S$,
  corresponding to a relative line $L\subset Z \ra S$.  Let $G_L \ra
  S$ be the Grassmann scheme of relative two-planes in $\proj V$ which
  contain $L$; note that for $s \in S$, $(G_L)_s \iso \proj^2_s$.  Let
  $\Delta_L \subset G_L$ be the (discriminant) subscheme of planes
  which meet $Z$ in a union of three lines.  For $s \in S$,
  $(\Delta_L)_s = \Delta_{L_s}$, and thus $\Delta_L \ra S$ is a proper
  smooth irreducible relative curve of genus $6$.  Similarly, let
  $\til \Delta_L \subset F$ be the subscheme of lines which meet $L$.
  Then $\til \Delta_L \ra \Delta_L$ is an \'etale morphism of degree
  two, and $\til \Delta_L \ra S$ is a proper smooth irreducible
  relative curve of genus $11$.   By Lemma \ref{lemprym}, 
there is a principally polarized complement $P(Z)$ for
$\pic^0(\Delta_L/S)$ inside 
  $\pic^0(\til\Delta_L/S)$; $P(Z)$ is the
  sought-for abelian scheme over $S$.

  Now suppose that $F^\good \ra S$ does not admit a section.  Since
  $S$ is normal, every connected component of $S$ is irreducible, and
  we may assume that $S$ itself is irreducible.  Since $F^\good \ra S$
  is smooth, it admits a section after \'etale base change.  So, let
  $T \ra S$ be an \'etale Galois morphism for which there exists a
  section $L \in F^\good(T)$, corresponding to a line in $\proj V$
  contained in $Z_T$.  The principally polarized abelian scheme $P_L
  \ra T$ has already been constructed.

The relevant descent argument is now entirely similar to that used in
Lemma \ref{lemarbfield}. Let $\eta$ be the (disjoint) union of the generic points
of $T$; note that each of these points is isomorphic to every other.  By Lemma
\ref{lemarbfield}(b), for each $\sigma \in \aut(T/S)$ there is a
canonical isomorphism $P_{\sigma L,\eta} \sriso P_{L,\eta}$.  Since $T$
is normal, each of these extends to an isomorphism $P_{\sigma L} \sriso
P_L$ \cite[Prop.\ I.2.7]{faltingschai}.  By Galois descent \cite[Sec.\
6.2.B]{blr}, $P_L$ descends to $S$.
\end{proof}

\begin{corollary}
\label{corperiodthreefold}
There is a morphism $\til\varpi:\tilstk T \ra \stk A_5$ such that, for
$t\in \tilstk T$, $\stk X_{\til\varpi(t)}\iso P(\stk Z_s)$.
\end{corollary}

\begin{proof}
The parameter space $\tilstk T$ is itself represented by a normal
Noetherian scheme; now use Theorem \ref{thbuildprym}.
\end{proof}

Deligne has conjectured \cite[3.3]{deligneniveau} the existence of
such an arithmetic Torelli map for any universal smooth complete
intersection of Hodge level one.  In fact, he
proves that the universal intermediate Jacobian attached to the
complex fiber descends to $\rat$, and conjectures that it spreads out to $\integ$.
Corollary \ref{corperiodthreefold} proves this conjecture for the
universal family of cubic threefolds, at least away from the prime
$2$.  With this fact supplied, work of Deligne now gives a way to
calculate the middle cohomology of the tautological cubic threefold:

\begin{proposition}
\label{propdeligne}
Let $\calp = \til\varpi^*\calx$ be the relative Prym variety over
$\tilstk T$, with structural map $j:\calp \ra \tilstk T$.  Recall that
$h:\calz \ra \tilstk T$ is the tautological smooth cubic threefold.  
\begin{alphabetize}
\item For each rational prime $\ell$, 
  there is an isomorphism of sheaves  $R^3h_*\integ_\ell(1) \sriso
  R^1j_*\integ_\ell$ on $\tilstk T_{\integ[1/2\ell]}$ compatible with the intersection forms;
\item the canonical isomorphism $H^3_\derham(\calz/\tilstk T_\cx)
  \sriso H^1_\derham(\calp/\tilstk T_\cx)$ of vector bundles with
  connection descends to $H^3_\derham(\calz/\tilstk T_\rat) \sriso
  H^1_\derham(\calp/\tilstk T_\rat)$; and
\item for each odd prime $p$, there is a canonical isomorphism
  $H^3_\cris(\calz /\tilstk T_{\ff_p}) \sriso H^1_\cris(\calp /\tilstk
  T_{\ff_p})$ of isocrystals on $\tilstk T_{\ff_p}$.
\end{alphabetize}
\end{proposition}

\begin{proof}
We specialize the arguments of Deligne, which
are valid for any universal smooth complete intersection of level one,
to the case of cubic threefolds.

Part (a) is \cite[Prop.\ 3.4]{deligneniveau}.  

For part (b), the
canonical isomorphism of vector bundles on $\tilstk T_\cx$ may be
described as follows.  The transcendental construction of the
intermediate Jacobian yields an isomorphism 
\begin{diagram}
R^3h_{*}\integ(1)&\rto^\twiddle 
&
R^1j_*\integ
\end{diagram}
of polarized variations of Hodge structure on $\tilstk
T_\cx$.  Tensoring with the structure sheaf gives an isomorphism
\begin{diagram}
H^3_\derham(\calz/\tilstk T_\cx) &\rto^{\alpha_\cx} &
H^1_\derham(\calp/\tilstk T_\cx).
\end{diagram}
This isomorphism is compatible with the Gauss-Manin connection, and
takes the cup product on 
$H^3_\derham(\calz/\tilstk T_\cx)$ to the polarization form on
$H^1_\derham(\calp/\tilstk T_\cx)$.  Abstractly, one knows that there is
{\em some} horizontal isomorphism $\beta:H^3_\derham(\calz/\tilstk T_\rat)
\sriso H^1_\derham( \calp/\tilstk T_\rat)$ of vector bundles on
$\tilstk T_\rat$; Deligne conjectures
\cite[Conj.\ 2.10]{deligneniveau} that $\alpha_\cx$ itself
descends.   This would certainly be true if $\alpha_\cx$ were induced
by an algebraic correspondence on $\calz_\rat \cross_{\tilstk T_\rat}
\calp_\rat$, as predicted by the Hodge conjecture.

Deligne further gives several equivalent formulations of this
conjecture \cite[p.\ 247]{deligneniveau}.  The isomorphism $\beta$
takes the cup product to some multiple
$\mu(\beta)\in\rat\units$ of the polarization form.  This multiple is
necessarily constant on all of $\tilstk T_\rat$, and the class of
$\mu(\beta)$ in $\rat\units/(\rat\units)^2$ is independent of the
actual choice of ($\rat$-multiple of) $\beta$.  The isomorphism
$\alpha_\cx$ descends to an isomorphism of vector bundles on $\tilstk
T_\rat$ if and only if $\mu(\beta)$ is a square.

Since $\mu(\beta)$ is constant on $\tilstk T_\rat$, to prove (b) it
now suffices to prove the Hodge conjecture for a single member of the
family $\calz_\rat\cross_{\tilstk T_\rat} \calp_\rat  \ra \tilstk
T_\rat$.  Let $s\in \tilstk T_\rat$ represent the Fermat cubic
threefold, and let $X_0$ be the elliptic curve with complex
multiplication by $\integ[\zeta_3]$.  The intermediate
Jacobian $\calp_s$ is isogenous to $X_0^{\oplus 5}$
(\cite[11.9]{act02}, \cite{carlsontoledospecial}). Consequently,
$\calz_s \cross \calp_s$ is dominated by a product of Fermat
hypersurfaces of (prime) degree three, and thus satisfies the Hodge
conjecture \cite[Thm.\ IV]{shiodafermathodge}. 

Part (c) follows immediately from (b); after tensoring coefficients
with $\rat_p$, the crystalline cohomology of a
smooth proper family $Y \ra S/\ff_p$, thought of as a module with
connection,  may be computed using the de Rham
cohomology of a smooth proper lift to characteristic zero \cite[Thm.\ 3.10]{ogusii}.
\end{proof}

\subsection{Prym schemes}
\label{subsecprym}

Theorem \ref{thbuildprym} proceeds through a Prym
construction.  It is presumably well-known that one can associate a
principally polarized abelian scheme to any \'etale double cover of
relative curves over a base on which $2$ is invertible.  For want of a
suitable reference, details are provided here.

Let $X \ra S$ be an abelian scheme over a connected base, with dual
abelian scheme $\hat X$.  A line
bundle $\call$ on $X$ defines a morphism $\phi_\call: X \ra \hat X$.
A polarization of $X$ is an isogeny $\lambda:X \ra \hat X$ which,
\'etale-locally on $S$, is of the form $\phi_\call$ for some ample
$\call$.

Now suppose that $Y$ is a sub-abelian scheme of $X$, with inclusion
map $\iota:Y \inject X$.  A polarization $\lambda$ on $X$ induces a
polarization $\lambda_Y$ of $Y$.  As a map of abelian schemes, it is
given by $Y \sra \iota X \sra \lambda \hat X \sra{\hat\iota}
\hat Y$.  If $\lambda = \phi_\call$ for some ample line bundle
$\call$, then $\iota^*\call$ is ample on $Y$ and $\lambda_Y =
\phi_{\iota^*\call}$. Since $\lambda_Y$ is a surjective map of smooth group
schemes, it is flat \cite[Lem.\ 6.12]{git}.  In particular,
$\ker\lambda_Y$ is a finite flat group scheme.  The exponent of $Y$
(as a sub-abelian scheme of the polarized abelian scheme $X$) is the
smallest natural number $e = e(Y\inject X, \lambda)$ which annihilates
$\ker\lambda_Y$. 

\begin{lemma}
\label{lemcomp}
Let $\iota:Y\inject X \ra S$ be an inclusion of abelian schemes over a
connected base.  Let $\lambda$ be a principal polarization,
and suppose that $e = e(Y\inject X,\lambda)$ is invertible on $S$.
Then there is a complement for $Y$ inside $X$.
\end{lemma}

\begin{proof}
Let $[e]_Y$ denote the multiplication-by-$e$ map on $Y$.  Since $\ker
\lambda_Y \subseteq \ker [e]_Y$, there exists a morphism $\mu_Y:\hat Y
\ra Y$ such that $\mu_Y \comp \lambda_Y = [e]_Y$.  
Consider
the norm endomorphism $N_Y \in \End(X)$ given by
\begin{diagram}
X & \rto^\lambda & \hat X & \rto^{\hat\iota} & \hat Y & \rto^{\mu_Y} &
Y & \rto^\iota & X.
\end{diagram}
The image $N_X(X)$ is $Y$.  Moreover, the argument of  \cite[p.\
125]{birkenhakelange} shows that $N_X\rest Y = [e]_Y$.     (While the
cited result is only claimed for complex abelian varieties, the key
calculation \cite[Lem.\ 5.3.1]{birkenhakelange} relies only on the
symmetry of the polarization morphism $\lambda$.)  

Consider the morphism $M_Y := \mu_Y\comp \hat\iota \comp \lambda:X \ra
Y$ of abelian schemes over $S$.  We will show that $M_Y$ is smooth.
Since $M_Y$ is a surjective morphism of smooth group schemes, it is
flat.  Consequently, in order to show $M_Y$ is smooth, it suffices
to show that for each point $y \in Y$, the morphism $X_y \ra \spec
\kappa(y)$ is smooth, where $\kappa(y)$ is the residue field at $y$
\cite[IV$_4$17.5.1]{ega4}.  Thus, it suffices to prove the claim in
the special case where $S = \spec K$ is the
spectrum of a field in which $e$ is invertible.   Because $M_Y$ is a
homomorphism of group schemes over a field, it suffices to verify its
smoothness at a single point \cite[VI.1.3]{sga3}.
The respective images $\zz_X$
and $\zz_Y$ of the identity sections of $X$ and $Y$ are each, as
$S$-schemes, isomorphic to $\spec K$ itself.  Therefore, in order to
show $M_Y$ is smooth, it suffices to show that the induced map on
tangent spaces $T_X(\zz_X) \ra T_Y(\zz_Y)$ is surjective \cite[IV$_4$17.11.1]{ega4}.  This last
claim is now obvious; because $e$ is invertible in $K$, the
differential $[e]_*:T_Y(\zz_Y) \ra T_Y(\zz_Y)$ is already surjective.

In particular, $M_Y:X \ra Y$ is separable, and thus its kernel $Z :=
X\cross_Y \zz_Y$ is a reduced group scheme over $S$.  Its connected
component of identity is the sought-for abelian scheme.
\end{proof} 

\begin{lemma}
\label{lemprym}
Let $S$ be a scheme on which $2$ is invertible, and let $\til
C \ra C/S$ be an \'etale double cover of smooth proper relative curves
over $S$.  Then there is a principally polarized complement for
$\pic^0(C)$ inside $\pic^0(\til C)$.
\end{lemma}

\begin{proof}
The relative Picard scheme $\pic^0(\til C)$ is an abelian scheme over
$S$ with canonical principal polarization $\lambda$.  The cover of
curves yields a canonical 
inclusion $\pic^0(C) \inject \pic^0(\til C)$, and the exponent of
$\pic^0(C)$ is $2$.  (The exponent may be computed at any (geometric)
point of each component of $S$, in which case the result is classical
\cite[Cor.\ 1]{mumfordprym}.)   By Lemma \ref{lemcomp}, there exists a
complement $Z$ to $\pic^0(C)$ inside $\pic^0(\til C)$.   
Moreover, the polarization on
$Z$ induced by $\lambda$ is twice a principal polarization
\cite[Cor.\ 2]{mumfordprym}.
\end{proof}

\section{Cubic surfaces}
\label{secsurfaces}

We now resume working in the category of schemes over $\calo_E[1/6]$.
As noted in the introduction, to a cubic surface one can associate a
cubic threefold.  By composing this morphism with $\til\varpi:\tilstk T
\ra \stk A_5$ (Corollary \ref{corperiodthreefold}), one obtains a morphism $\tilstk S \ra \stk A_5$.  On
one hand, the abelian 
fivefolds thus obtained have an action by $\integ[\zeta_3]$, and the image of
this morphism is contained in $\stk M$.  On the other hand,
the main result of the present section is that this morphism actually
descends to a morphism of stacks $\stk S \ra \stk M$ (Proposition
\ref{propperiodsurface}) which is injective on points.  This is
accomplished by introducing rigidifying structures on both sides; a
choice of marking of the 27 lines on a cubic surface is tantamount to
a level $(1-\zeta_3)$-structure on the associated abelian variety.

\subsection{Action of $\integ[\zeta_3]$ on $P(Y)$}

Consider the morphism $\til\phi:\til S \ra \tilstk T$
which sends a cubic form in four variables $f(X_0,X_1,X_2,X_3)$ to the
cubic form $X_4^3 - f(X_0,X_1,X_2,X_3)$ in five variables.  (This
construction is introduced in \cite[Sec.\ 2.1]{act02}, albeit over $\cx$.)
Let
$j:\SL_4 \ra \SL_5$ be the inclusion $A \mapsto {\rm diag}(A,1)$.  Then
the diagram
\begin{diagram}
\SL_4 \cross \tilstk S & \rto &\tilstk S \\
\dto<{j\cross \til\phi}&&\dto>{\til\phi} \\
\SL_5 \cross \tilstk T & \rto & \tilstk T
\end{diagram}
commutes, and $\til\phi$ descends to a morphism of stacks
\begin{diagram}
\stk S & \rto^\phi & \stk T.
\end{diagram}
This morphism sends a cubic surface $Y$ over an affine scheme $\spec A$ to the
cyclic degree-three cover $Z = F(Y)$ of $\proj^3_A$ which is
ramified over $Y$. (In the general case of a cubic surface $Y$ inside a
projective $S$-bundle $\proj V$, $F(Y)$ is defined by glueing on
the base.)
We will typically denote
this cover by $\pi_Y: Z \ra \proj^3$. Note that, by construction,
$\aut(F(Y)/\proj^3)\iso \integ/3$, say with generator $\gamma$.

Moreover, suppose $f$ is a strictly stable form; it has at least one
singularity, and each such looks locally like $x_1^2+x_2^2+x_3^2=0$.
Then each singularity of $X_4^3 - f$ looks locally like
$x_4^3-(x_1^2+x_2^2+x_3^2)=0$, and in particular $X_4^3 - f$ is
strictly stable.  Consequently, $\til\phi$ and $\phi$ extend to
morphisms $\tilstk S_\stable \ra \tilstk T_\stable$ and $\stk
S_\stable \ra \stk T_\stable$, respectively.

\begin{lemma} 
\label{lemeisaction}
Suppose $S$ is a normal Noetherian scheme and $Y \in \stk S(S)$.  Then 
$P(F(Y))$ admits  
multiplication by $\integ[\zeta_3]$.
\end{lemma} 

\begin{proof}
Let $X/S$ be any
abelian scheme, and let $U\subset S$ be a nonempty open subscheme.
Since endomorphisms of $X\rest U$ extend to $X$, it suffices to prove
the result when $S = \spec K$ is the spectrum of a field.  Let $Z =
F(Y)$, and let $\pi_Y:Z \ra \proj^3$ be the cyclic cubic cover of
$\proj^3$ with branch locus $Y$.

The automorphism $\gamma$ induces an automorphism $\gamma^*$ of
$A^2(Z_{\bar K})$.  Moreover, $\gamma^*$ acts nontrivially, for
otherwise $\pi_Y^*:A^2(\proj^3_{\bar K}) \ra A^2(Z_{\bar K})$ would be
an isomorphism.  The universal property of $P(Z_{\bar K})$, applied to
$\alpha \comp \gamma^*: A^2(Z) \ra P(Z_{\bar K})$, shows that $\gamma$
induces a nontrivial automorphism of $P(Z_{\bar K})$ of order three.
By descent, $\gamma$ induces an automorphism of $P(Z)$, and
$\integ[\zeta_3] \subseteq \End_K(Z)$.

It remains to check that $1 \in \integ[\zeta_3]$ acts as $[1]_{P(Z)}
\in \End(P(Z))$.  Any point in $\tilstk T$ of positive
characteristic is the specialization of a point in characteristic
zero.  Since the characteristic polynomial of an endomorphism of an
abelian variety is constant in families, it suffices to verify that
$1$ acts as $[1]_{P(Z)}$ when $K$ has characteristic zero.  This claim
follows from, e.g., Proposition \ref{propdeligne} and
\cite[2.2]{act02}.
\end{proof}

\begin{remark}
For a suitably generic cubic surface $Y$ over an algebraically closed 
field $k$, it is possible to visualize the action of $\gamma$, as
follows.  If $Y$ is generic, then if a plane in
$\proj^3$ intersects $Y$ in two lines, then it
does so in three distinct lines. In particular, no plane in $\proj^3$
tangent to $Y$ is tangent along an entire line of $Y$.  So, let $L_0
\subset Y$ be one of the 27 lines, and let $L = \pi_Y\inv(L_0) \subset
Z$.  Since $\pi_Y$ is ramified over $Y$, $L$ is a line inside $Z$, and
in fact $L \in F_Z^\good(k)$.  Moreover, $\gamma$ acts on $F$ (by
pullback) and fixes $L$ (since $L$ is supported in the ramification
locus of $\pi_Y$).  Then $\til \Delta_L$ and $\Delta_L$ are each stable
under the action of $\gamma$, and $\integ[\zeta_3]$ acts on each of
$\pic^0(\til \Delta_L)$ and $\pic^0(\Delta_L)$. The map $\til \Delta_L
\ra \Delta_L$ is $\ang \gamma$-equivariant, and $\integ[\zeta_3]$ acts
on the Prym variety $P(\til \Delta_L/\Delta_L)$.
\end{remark}

\begin{lemma}
\label{lemsig}
If $S$ is a normal Noetherian scheme and if $Y \in \stk S(S)$, then
the signature of $P(F(Y))$ is $(4,1)$. 
\end{lemma}

\begin{proof}
Since the signature of an $\calo_E$ action is constant in families (on
which the discriminant of $\calo_E$ is invertible), and since $\stk S$
is irreducible, it suffices to verify the claim at a single point.
A direct Hodge-theoretic calculation \cite[Lem.\ 2.6]{act02} shows that the Prym variety
associated to any complex cubic surface has signature $(4,1)$.
\end{proof}

\subsection{Markings and level structure}

Suppose that $k$ is algebraically closed and $Y \in \stk S(k)$, and as
usual let $\pi_Y:Z= F(Y) \ra \proj^3$ be the cyclic cubic cover of $\proj^3$
ramified along $Y$.  If $L_1$ and $L_2$ are lines in $Y$, then the cycle
$[L_1]-[L_2]$ is algebraically trivial in $\proj^3$, and thus
$\pi_Y^*([L_1]-[L_2]) \in A^2(Z)$.  Let $\til L_i = \pi_Y\inv(L_i)$
for $i \in \st{1,2}$.  Then $\pi_Y^*[L_i] = 3[\til L_i]$.   Since
$A^2(Z)$ is a divisible group \cite[Lem.\ 0.1.1]{beauville77}, 
$[\til L_1] - [\til L_2] \in A^2(Z)$.  Moreover, since each $L_i$ is 
supported in the ramification locus of $\pi_Y$, $\til L_1$ and $\til L_2$
are fixed by $\gamma^*$.  By functoriality of the Prym construction, $\alpha([\til
L_1] - [\til L_2]) \in P(Z)[1-\zeta_3](k)$, the kernel on $P(Z)$ of
multiplication by $1-\zeta_3$.

More generally, suppose $Y/S$ is a smooth cubic surface over a scheme
$S$.  If $L_1$ and $L_2$ are relative lines in $Y$, then the difference
$[\til L_1] -
[\til L_2]$ of the classes of their inverse images under $\pi_Y$ corresponds to a section of $P(Z)[1-\zeta_3]$.

As in \cite[p.755]{matsumototerasoma}, fix isomorphisms $\weyl \iso
\aut(\Lambda_0)$ and $\weyl \iso \po(V_0,q_0)$.

\begin{lemma}
\label{lemmarkedmorphism}
Suppose $(Y,\Psi) \in
\stk S^\marked(S)$, and let $P = P(F(Y))$.  For $1 \le i \le 5$, let
$v_i = \alpha([\Psi\inv(e_i)] - [\Psi\inv(\ell_{i6})]) \in P(S)$.
Then $\st{v_1, \cdots, v_5}$ is an orthonormal basis for
$P[1-\zeta_3](S)$. 
\end{lemma}

\begin{proof}
Since $(Y,\Psi)$ is the pullback of the universal marked cubic surface
over $\stk S^\marked$, and since $\stk S^\marked$ is itself a normal Noetherian
quasiprojective scheme, we may assume that $S$ is normal Noetherian,
and even irreducible.
Since $3$ is invertible on $S$, the group scheme $P[1-\zeta_3]$ is
\'etale over $S$.  Therefore, it suffices to show there is some point
$s \in S$ such that $\st{v_{1,s}, \cdots, v_{5,s}}$ is an orthonormal
basis for $P[1-\zeta_3]_s(s)$.  Since $\stk S \ra \calo_E[1/6]$ is
flat, possibly after replacing $S$ with a lift to characteristic
zero, we may assume that $S$ has a point whose residue field has
characteristic zero.  The desired result is then \cite[Prop.\ 5.5]{matsumototerasoma}.
\end{proof}

\begin{proposition}
\label{propperiodsurface}
\begin{alphabetize}
\item 
There is a $\weyl$-equivariant radicial morphism
\begin{diagram}
\tau^\marked:\stk S^\marked & \rto & \stk M^{(1-\zeta_3)} \\ 
\end{diagram}
of schemes over $\calo_E[1/6]$.

\item
There is a morphism
\begin{diagram}
\tau: \stk S & \rto & \stk M
\end{diagram}
of stacks over $\calo_E[1/6]$ which is injective on points
and specializes to the complex period map of \cite{act02}.
\end{alphabetize}
\end{proposition}

\begin{proof}
The essential content is already present in Lemma \ref{lemmarkedmorphism}.
If $S$ is a normal
Noetherian scheme and $(Y,\Psi)$ is a marked cubic surface over $S$,
then the image of the moduli point of $(Y,\Psi)$ is that of the
abelian scheme $P(F(Y))$ with 
its canonical polarization and $\calo_E$-action, with level structure $\Phi: v_i
\mapsto \alpha([\Psi\inv(e_i)] - [\Psi\inv(\ell_{i6})])$. By applying
this construction to $\stk S^\marked$ itself we obtain a
$\stk S^\marked$-valued point of $\stk M^{(1-\zeta_3)}$, i.e., a morphism
$\stk S^\marked \ra \calm^{(1-\zeta_3)}$.   The $\weyl$-equivariance is
worked out in detail by Matsumoto and Terasoma; see \cite[Sec.\ 3.2,
Prop.\ 5.5]{matsumototerasoma}.  Beauville's Torelli-type
theorem \cite[Cor.\ on p.205]{beauville81} shows that this morphism is injective on geometric
points. Therefore, $\tau^\marked$ is radicial, i.e., injective on
$K$-points for any field $K$ admitting a map $\calo_E[1/6] \ra K$.  This proves (a). 

For (b), use part (a) and the identifications $\stk S =
\stk S^\marked/\weyl$ and $\stk M = \stk M^{(1-\zeta_3)}/\weyl$.  The
fact that $\tau_\cx$ specializes to the period map of \cite{act02}
follows from the known isomorphism
(Section \ref{subsecprymalgclosed}) between the Prym $P(Z)$ and the 
intermediate Jacobian $J(Z)$ of a smooth complex cubic threefold.
\end{proof}

Thus, we have constructed morphisms $\til\varpi$ and $\tau$ as follows:
\begin{diagram}[LaTeXeqno]
\label{diagbig}
\tilstk S & \rto^{\til \phi} & \tilstk T \\
\dto && \dto &\rdto>{\til\varpi}\\
\stk S & \rto^\phi & \stk T &&\stk A_5 \\
&\rdto>\tau&&\ruto^\jmath &\\
&&\stk M
\end{diagram}
Since $\stk S$ and $\stk M$ have irreducible four-dimensional fibers
over $\spec \calo_E[1/6]$, and since $\tau$ is radicial, it follows
that the image of $\tau$ is open in $\stk M$.  Describing this image
precisely requires a consideration of the compactification of $\tau$.

\section{Characterization of the image}
\label{secimage}

 The
goal of the present section is to extend $\tau$ to a homeomorphism
$\stk S_\stable \ra \stk M$ (Theorem \ref{thbigdiagram}), and thus
give an arithmetic version of many of the complex analytic facts
established in \cite{act02}.  In
particular, the complement of $\tau(\stk S)$ in $\stk M$ is an
irreducible horizontal divisor which corresponds to cubic surfaces
which are stable but not smooth.

\begin{lemma}
\label{lemcompdivisor}
The complement $\tilstk S_\stable \setcomp \tilstk S$ is an
irreducible horizontal 
divisor in $\tilstk S_\stable$.
\end{lemma}

\begin{proof}
  We need to show that if $s \in \spec \calo_E[1/6]$, then 
  $(\tilstk S_\stable \setcomp \tilstk S)_s$ is an irreducible divisor in
  $\tilstk S_{\stable,s}$.  This is \cite[Prop.\ 6.7]{beauville09}; while
  the result is claimed only for $(\tilstk S_\stable\setcomp\tilstk S)_\cx$, the
  argument given there is valid over an arbitrary base field.
\end{proof}

Let $\til\tau$ be the composition $\tilstk S \ra \stk S \sra \tau \stk M$.

\begin{lemma} 
\label{lemtiltauextends}
The Torelli map $\til\tau:\tilstk S \ra \stk M$ extends to $\til\tau:\tilstk S_\stable \ra
\stk M$. 
\end{lemma}

\begin{proof}
The existence of $\tilstk S_{\stable,\cx} \ra \stk M_\cx$ is \cite[Thm.\
3.17]{act02}.  By descent (e.g., \cite[15.8]{kudlarapoportii}) one
obtains $\tilstk S_{\stable, E} \ra \stk M_E$.

Now let $\idp \subset \calo_E[1/6]$ be a nonzero prime ideal, and let
$\calo_{E,(\idp)}$ be the localization at $\idp$; it is an unramified
discrete valuation ring of mixed characteristic.  Let $[\idp]$ be the
closed point of $\spec \calo_{E,(\idp)}$.  Finally, let $S =
\tilstk S_\stable \cross \spec \calo_{E,(\idp)}$, and let $T =
(\tilstk S_\stable  \setcomp \tilstk S)_{[\idp]}$.  Then $S$ is
smooth, and $T$ is a 
closed subscheme of codimension two supported over the closed point of
$\spec \calo_{E,(\idp)}$.  A result of Faltings \cite[Lemma
3.6]{moonen96} shows that the abelian scheme over $S\setcomp T$
extends uniquely to one 
over all of $S$. Since the $\calo_E$-action also extends to this
abelian scheme, and the signature is constant, we obtain a morphism $S \ra \stk
M$.  Glueing yields the desired extension
$\tilstk S_\stable \ra \stk M$.
\end{proof}

\begin{lemma} 
The marked Torelli map $\tau^\marked: \stk S^\marked \ra \stk
M^{(1-\zeta_3)}$ extends to $\stk S_\stable^\marked \ra
\stk M^{(1-\zeta_3)}$ and to $\stk S_\semistable^\marked \ra \barstk
M^{(1-\zeta_3)}$.  The Torelli map $\tau:\stk S \ra \stk M$ extends to
$\stk S_\stable \ra \stk M$ and to $\stk S_\semistable \ra \barstk M$.
\end{lemma}

\begin{proof}
As in Section \ref{subsecmodcubics}, let $\coarse S_\stable$ be the
underlying coarse moduli scheme 
of $\stk S_\stable$, and similarly define $\coarse M$.  (Recall that
$\stk S^\marked_\stable$ is itself already a scheme.)
Since the morphism $\til
\tau: \tilstk S_\stable \ra \stk M$ of Lemma \ref{lemtiltauextends} is constant on the fibers of $\tilstk S_\stable \ra \stk
S_\stable$, we have a map of normal quasiprojective schemes $\coarse
\tau: \coarse S_\stable \ra \coarse M$.  Via the projection
$\stk S^\marked_\stable \ra \coarse S_\stable$, we obtain a dominant
morphism $\psi: 
\stk S^\marked_\stable \ra \coarse M$ which, as a rational map,
factors through $\stk M^{(1-\zeta_3)}$.  Therefore, as a morphism of schemes,
$\psi$ factors through the normalization of $\coarse M$ in $\stk
M^{(1-\zeta_3)}$; since $\stk M^{(1-\zeta_3)}$ is already a normal scheme,
$\psi$ factors through $\tau^\marked:\stk S_\stable^\marked \ra \stk
M^{(1-\zeta_3)}$.

This proves the first part of the claim.  The extension of
$\tau^\marked$ to $\stk S_\semistable$ follows from \cite[IV$_4$.20.4.12]{ega4}; by
choosing an affine neighborhood of each 
cusp, we have a rational map defined on an open subset of a normal
affine scheme whose complement has codimension greater than one.

Finally, the extension of $\tau$ follows from the $\weyl$-equivariance
of $\tau^\marked$.
\end{proof}

Ultimately, it will turn out that $\tau$ is injective on points.  As a
preliminary step, we will show that $\tau$ preserves the
stratification $\stk S \subsetneq \stk S_\stable \subsetneq S_\semistable$.

\begin{lemma}
\label{lemthetadivisor}
We have $\tau(\stk S) \cap \tau(\stk S_\stable \setcomp \stk S) = \emptyset$
and $\tau(\stk S_\stable) \cap \tau(\stk S_\semistable \setcomp \stk
S_\stable) = \emptyset$.
\end{lemma}

\begin{proof}
The unique point of $(\stk S_\semistable \setcomp \stk S_\stable)(\cx)$
is sent to the unique cusp of $\barstk M_\cx$, and thus the same is true for
every geometric fiber of $\stk S_\semistable \ra \spec \calo_E[1/6]$.
Since $\tau(\stk S_\stable) \subset \stk M$, the second claim follows
immediately.

The first claim will be achieved using the explicit nature of the Torelli
theorem for cubic threefolds \cite{beauville81}.  Suppose $Z/k$ is a
smooth cubic threefold over 
an algebraically closed field, with principally polarized Prym variety
$(X,\lambda)$.  Consider the theta divisor $\Theta\subset X$
associated to $\lambda$.  Then $\Theta$ has a unique singularity,
which is a triple point; and $Z$ (actually, its affine cone) may be
reconstructed as the tangent cone at that singularity.  

Moreover, if $k=\cx$, a converse is available \cite{cmf05}; a
principally polarized abelian fivefold $(X,\lambda)$ is the Prym
variety of a smooth cubic threefold if and only if
its theta divisor has a unique singularity, which is a singularity of
order three.

In fact, the theta divisor of a stable, but not smooth, cubic
threefold must be more singular than that of a smooth cubic
threefold, as follows.

Let $Y/\cx$ be a nodal stable cubic surface.  Then $Z=F(Y)$ is a nodal
stable cubic 
threefold.  Because $Z$ can be realized as a degeneration of smooth cubic
threefolds, the theta divisor $\Theta$ of $(P(Z),\lambda)$ has at least one
singular point of order at least three.  In fact, either $\Theta$ has
a second singular point, or its unique singularity has order greater
than three; for if not, then by \cite{cmf05}, $(P(X),\lambda)$ would
also be the Prym of a smooth cubic threefold (still with suitable
$\integ[\zeta_3]$-action), which would contradict the known
injectivity of $\tau_\cx$ on $\stk S_{\stable,\cx}$ \cite[Thm.\ 3.17]{act02}.

Similarly, let $Y/k$ be a nodal stable cubic surface over a field
of characteristic at least $5$.  Let $R$ be a mixed characteristic
discrete valuation ring with residue field $k$; let $K = \Frac R$ be its
fraction field.  Lift $Y$ to a cubic surface $\til Y/R$ such that
$\til Y_K$ is also nodal stable.  We have seen that the theta divisor of the
principally polarized abelian variety $P(F(\til Y_K))$ either has at least
two singular points, or has one singularity of order greater than
three.  The same is necessarily true of $P(F(\til Y_k)) = P(F(Y))$, its
specialization.

In summary, the geometry of the theta divisor means that the Prym
variety of a nodal stable cubic surface cannot also be the Prym
variety of a smooth cubic surface.
\end{proof}

Let $\stk N$ be the image of $\stk S$ under $\tau$, and let $\stk D$
be its complement.  Then $\stk N^{(1-\zeta)}:= \stk N \cross_{\stk M}
\stk M^{(1-\zeta)}$ is the image of $\stk S^\marked$ under
$\tau^\marked$, with complement $\stk D^{(1-\zeta)}$.

\begin{lemma}
\label{lemsmoothmarkedisom}
The marked Torelli morphism maps $\stk S^\marked$ isomorphically onto
its image $\stk N^{(1-\zeta)}$.
\end{lemma}

\begin{proof}
Because $\tau^\marked\rest{\stk S^\marked}$ is radicial (Proposition \ref{propperiodsurface}) and of
finite type, it is quasifinite. Since $\stk S_\semistable^\marked$ is
projective, $\tau^\marked$ is projective.  Lemma \ref{lemthetadivisor} shows that
$\stk S_\semistable^\marked \cross_{\stk M^{(1-\zeta)}} \stk
N^{(1-\zeta)}$ is $\stk S^\marked$ (and
not larger), and so $\tau^\marked \rest{\stk S^\marked}$ is
proper, thus finite.

Now, $\stk N^{(1-\zeta)}$ is smooth, and in particular its local rings
are unique factorization domains.  By computing locally on $\stk
N^{(1-\zeta)}$, we see that the finite morphism $\tau^\marked
\rest{\stk S^\marked}$ is necessarily flat \cite[Thm.\
22.6]{matsumuracommring}.  In particular, its  
degree is constant.  Since that degree is known to be one in
characteristic zero \cite[Thm.\ 3.17]{act02}, $\tau^\marked \rest{\stk S^\marked}$ is an
isomorphism onto its image.
\end{proof}

\begin{lemma}
\label{lemsemistablemarkedisom}
The marked Torelli morphism $\tau^\marked: \stk S_\semistable^\marked \ra
\barstk M^{(1-\zeta_3)}$ is an isomorphism.
\end{lemma}

\begin{proof}
Each of $\stk S_\semistable^\marked \ra \spec \calo_E[1/6]$
and $\barstk M^{(1-\zeta)} \ra \spec \calo_E[1/6]$ has
four-dimensional geometric fibers.  Lemma \ref{lemsmoothmarkedisom}
implies that, fiberwise on $\spec \calo_E[1/6]$, $\tau^\marked$ is a
birational morphism.  Moreover, over the generic
fiber of $\spec \calo_E[1/6]$, $\tau^\marked$ is an isomorphism
\cite[Thm.\ 3.17]{act02}.  In particular, $\tau^\marked$ is a
birational morphism.

Let $\sigma$ be the
inverse of $\tau^\marked$ in the category of rational maps.  Suppose
$\idp\subset \calo_E[1/6]$ is a nonzero prime ideal.  The locus of
indeterminacy of $\sigma_{[\idp]}$ has positive codimension in
$\barstk M^{(1-\zeta_3)}_{[\idp]}$, and thus the locus of indeterminacy of
$\sigma$ on $\barstk M^{(1-\zeta_3)}\cross \spec \calo_{E,(\idp)}$ has
codimension at least two.  Since $\barstk M^{(1-\zeta_3)}\cross \spec
\calo_{E,(\idp)}$ is normal and $\tau^\marked$ is proper, $\sigma$ is
defined on all of $\barstk M^{(1-\zeta_3)}$ and $\tau^\marked$ is biregular. 
\end{proof}

We now come to the arithmetic analogue of the main
results of \cite{act02}.  In \cite{act02}, Theorem 3.20
states that the stacks $\stk S_{\stable,\cx}$ and $\stk M_\cx$ are
isomorphic, and Theorems 3.17 and 8.2 identify the GIT
compactification $\stk S_{\semistable,\cx}$ with the minimal
compactification $\barstk M_\cx$.  These assertions are actually
specializations of maps of stacks over $\integ[\zeta_3,1/6]$:

\begin{theorem}
\label{thbigdiagram}
  The Torelli morphisms $\tau$ and $\tau^\marked$ give rise
  to the commutative diagram of stacks over $\integ[\zeta_3,1/6]$
\begin{diagram}[tight,height=0.3in]
\stk S^\marked &\rto &&&\stk
N^{(1-\zeta_3)} \\
\dinject&\rdto&&&\diline&\rdto\\
&&\stk S & &\rto&\HonV & \stk N \\
&&\dinject &&\dto&&\dinject\\
\stk S^\marked_\stable &\hLine&\VonH&\rto&\stk M^{(1-\zeta_3)} \\
\dinject&\rdto&&&\diline&\rdto\\
&&\stk S_\stable & &\rto&\HonV & \stk M \\
&&\dinject &&\dto&&\dinject\\
\stk S^\marked_\semistable &\hLine & \VonH &\rto &\barstk
M^{(1-\zeta_3)} \\
&\rdto&&&&\rdto\\
&&\stk S_\semistable & &\rto& & \barstk M
\end{diagram}
in which all vertical arrows are the natural inclusions; all diagonal
arrows are quotients by the action of $\weyl$; all horizontal arrows
are homeomorphisms; the rear horizontal arrows are isomorphisms of
schemes; and the top front horizontal arrow is an isomorphism of stacks.
\end{theorem}

\begin{proof}
The assertions for the top (horizontal) square follow from Lemma
\ref{lemsmoothmarkedisom} and $\weyl$-equivariance; those for the
bottom, from Lemma \ref{lemsemistablemarkedisom}.  Every geometric
fiber of $\barstk M\setcomp \stk M$ consists of a single point, as does
each geometric fiber of $\stk S_\semistable \setcomp \stk S_\stable$.
Consequently, the image of $\stk S_\stable$ under $\tau$ is exactly
$\stk M$, which finishes the proof of the theorem.
\end{proof}

\begin{corollary} 
\label{corcomplement}
Let $\stk D$ be the complement of $\stk N$ in $\stk M$.  Then $\stk D$
is an irreducible horizontal divisor.
\end{corollary}

In fact, $\stk D$ is a special cycle in the sense of Kudla and
Rapoport \cite{kudlarapoportii}.

\begin{proof}
By Theorem \ref{thbigdiagram} and Lemma \ref{lemcompdivisor}, $\stk D$
is, fiberwise on $\spec \calo_E[1/6]$, a closed irreducible substack
of $\stk M$.  It therefore 
suffices to identify a closed horizontal irreducible substack of $\stk N$
which has codimension one in $\stk M$.  The Prym variety associated to
a smooth cubic threefold, and in particular to the triple cover of
$\proj^3$ ramified along a given smooth cubic surface, is irreducible
as a principally polarized abelian variety
\cite{beauville81,clemensgriffiths}.  Therefore, it suffices to
exhibit a horizontal divisor in $\stk M$ which parametrizes reducible
principally polarized abelian varieties with $\calo_E$-action.  Let
$\stk M_{(3,1)}$ be the moduli stack of abelian fourfolds with
action by $\calo_E$ of signature $(3,1)$, and let
$(X_0,\iota_0,\lambda_0)/\calo_E[1/6]$ be the (unique, principally
polarized) elliptic curve with action by $\calo_E$ of signature
$(1,0)$.  There is a closed immersion of stacks $\stk M_{(3,1)} \ra
\stk M$; on $S$-points, it is given by $(X,\iota,\lambda)\mapsto ((X_0\cross S)\cross X, \iota_0\cross \iota,
\lambda_0\cross\lambda)$.  Let $\stk D'$ be the image of this morphism.
Then every fiber of $\stk D' \ra \calo_E[1/6]$ has dimension $3$, and
$\stk D = \stk D'$ is the sought-for divisor.
\end{proof}

The modular interpretation of $\stk D$ is compatible with the
discussion of intermediate Jacobians of strictly stable cubic
threefolds in \cite{cmlaza09}.  A strictly stable cubic surface $Y$ has
a singularity of type $A_1$.  We have noted above that the cyclic
cubic threefold $F(Y)$ has a singularity of type $A_2$.  Consequently
(at least in characteristic zero), $P(F(Y))$ has an elliptic tail
\cite[Thm.\ 6.3]{cmlaza09}. The presence of the $\zeta_3$ action means
this tail must, in fact, be the elliptic curve
$(X_0,\iota_0,\lambda_0)$ introduced in Corollary \ref{corcomplement}.

Let $(X,\lambda)$ be a principally polarized abelian variety over
$\cx$ of dimension five.  Recall that Casalaina-Martin and Friedman
have shown \cite{cmf05} that $(X,\lambda)$ is the intermediate Jacobian of a smooth
cubic threefold if and only if its theta divisor has a unique
singularity, and that singularity has order three.  Their result
implies an analogous result for abelian varieties with
$\integ[\zeta_3]$-action, which is valid in all characteristics:

\begin{corollary}
Let $(X,\iota,\lambda) \in \stk M(k)$ be a principally polarized
abelian fivefold with action by $\integ[\zeta_3]$.   Then either the
theta divisor of $(X,\lambda)$ has a unique singularity, and that
singularity is of order three, in which case $(X,\lambda) = P(F(Y))$
for a smooth cubic surface $Y$; or the theta divisor has a
positive-dimensional singular locus, in
which case $(X,\lambda) = P(F(Y))$ for a nodal cubic surface $Y$.  In
each case, $Y$ is determined uniquely by $(X,\iota,\lambda)$.
\end{corollary}

\begin{proof}
By Theorem \ref{thbigdiagram}, if $(X,\iota,\lambda) \in \stk M(k)$,
then there is a unique $Y \in \stk S_\stable(k)$ such that $X =
P(F(Y))$.  On one hand, if $Y$ is smooth, then we have seen in the
proof of Lemma \ref{lemthetadivisor} that the theta divisor of
$(X,\lambda)$ has a unique singularity, and that singularity is of
order three.  On the other hand, if $Y$ is stable but not smooth, then
$(X,\iota,\lambda)$ is the product of an elliptic curve and an abelian
fourfold (Corollary \ref{corcomplement}), and thus the theta divisor
of $(X,\lambda)$ is singular along a positive-dimensional subvariety.
\end{proof}

\bibliographystyle{hamsplain}
\bibliography{jda}

\end{document}